\documentclass[12pt]{amsart}
\usepackage[english]{babel}
\usepackage{graphicx}
\usepackage{amsmath}
\usepackage{amsfonts}
\usepackage{amssymb}
\usepackage{latexsym}
\usepackage{amsthm}
\usepackage{hyperref}
\usepackage[T1]{fontenc}
\usepackage{color}
\usepackage{mathtools}
\usepackage{comment}

\input xy
\xyoption{all}

\newtheorem{thm}{Theorem}[section]
\newtheorem{cor}[thm]{Corollary}
\newtheorem{lem}[thm]{Lemma}
\newtheorem{pro}[thm]{Proposition}

\theoremstyle{definition}
\newtheorem{defi}[thm]{Definition}
\newtheorem{exe}[thm]{Example}

\newtheoremstyle{remarque}{}{}{}{}{\it}{.}{\newline}{}
\theoremstyle{remarque}
\newtheorem*{rem}{Remark}

\newcommand{\asd}[5]{%
\setbox1=\hbox{\ensuremath{^{#1}}}%
\setbox2=\hbox{\ensuremath{_{#2}}}%
\setbox5=\hbox{\ensuremath{#5}}%
\hspace{\ifnum\wd1>\wd2\wd1\else\wd2\fi}%
\ensuremath{\copy5^{\hspace{-\wd1}\hspace{-\wd5}#1\hspace{\wd5}#3}%
_{\hspace{-\wd2}\hspace{-\wd5}#2\hspace{\wd5}#4}%
}}

\DeclareMathOperator*{\rprod}{\mathrlap{\prod}{\coprod}}

\usepackage[OT2,T1]{fontenc}
\DeclareSymbolFont{cyrletters}{OT2}{wncyr}{m}{n}
\DeclareMathSymbol{\Sha}{\mathalpha}{cyrletters}{"58}
\DeclareMathSymbol{\Brusse}{\mathalpha}{cyrletters}{"42}

\newcommand{\z}{\mathbb{Z}}
\newcommand{\q}{\mathbb{Q}}

\newcommand{\ra}{\rightarrow}

\newcommand{\coker}{\mathrm{Coker}}
\renewcommand{\hom}{\mathrm{Hom}}
\renewcommand{\H}{\mathrm{H}}
\newcommand{\aut}{\mathrm{Aut}}
\renewcommand{\inf}{\mathrm{Inf}}
\newcommand{\res}{\mathrm{Res}}

\newcommand{\bic}{\mathrm{bic}}
\newcommand{\gm}{\mathbb{G}_{\mathrm{m}}}
\newcommand{\br}{\mathrm{Br}\,}
\newcommand{\un}{\mathrm{un}}
\newcommand{\tr}{\mathrm{tr}}
\newcommand{\brun}{{\mathrm{Br}_{\un}}}
\newcommand{\al}{\mathrm{al}}

\newcommand{\brunal}{\mathrm{Br}_{\un,\al}}
\newcommand{\sln}{\mathrm{SL}_n}
\newcommand{\inv}{\mathrm{inv}}
\newcommand{\bb}[1]{\mathbb{#1}}
\newcommand{\gal}{\mathrm{Gal}}
\newcommand{\Gal}{\mathrm{Gal}}
\newcommand{\bad}{\mathrm{Bad}}

\usepackage{marvosym}

\title[The Grunwald problem and approximation]{The Grunwald problem and approximation properties for homogeneous spaces}
\author{Cyril Demarche, Giancarlo Lucchini Arteche, Danny Neftin}

\date{}

\begin{document}

\begin{abstract}
Given a group $G$ and a number field $K$, the Grunwald problem asks whether given field extensions of completions of $K$ at finitely many places can be approximated by a single field extension of $K$ with Galois group $G$. This can be viewed as the case of constant groups $G$ in the  more general problem of determining  for which $K$-groups $G$ the variety $\sln/G$ has weak approximation. We show that away from an explicit set of bad places both problems have an affirmative answer for iterated semidirect products with abelian kernel. Furthermore, we give counterexamples to both assertions at bad places. These turn out to be the first examples of transcendental Brauer-Manin obstructions to weak approximation for homogeneous spaces. \\

{\bf MSC classes (2010):} 11R34, 14M17, 14G05, 11E72.
\end{abstract}

\maketitle

\section{Introduction}\label{sec:intro}
The Grunwald-Wang theorem has fundamental applications to the structure theory of finite dimensional semisimple algebras, cf.~\cite[Ch. 18]{Pie}, and provides an answer for abelian groups $G$ to the more general Grunwald problem. The latter is an inverse Galois problem of increasing interest due to its recently studied connections with the regular inverse Galois problem and with weak approximation, cf.~\cite{DG,HarariBulletinSMF}.

Fix a number field $K$,  let $K_v$ denote the completion of $K$ at a place $v$, and $\gal(K)$ denote its absolute Galois group. Let $G$ be a finite group, and $S$ a finite set of places of $K$. The Grunwald problem then asks whether every prescribed local Galois extensions $L^{(v)}/K_v$, $v\in S$, with embeddings $\gal(L^{(v)}/K_v)\hookrightarrow G$ can be approximated by a global extension $L/K$ with Galois group $G$. More precisely:

\bigskip

\noindent {\bf Grunwald Problem.} Is the restriction map
$$\hom(\gal(K),G)\ra \prod_{v\in S}\hom(\gal(K_v),G)/\sim $$
surjective?

Here $\phi_1\sim \phi_2$ if $\phi_1=g\phi g^{-1}$ for some $g\in G$. Note that the quotient by $\sim$ is necessary since the decomposition group of $L/K$ is defined up to conjugation.

\bigskip

Families of groups $G$ and number fields $K$ for which $(G,K,S)$ has an affirmative answer to the Grunwald problem for every $S$ include:
(1) abelian groups of odd order over every number field, by the Grunwald--Wang theorem \cite{Gru, Wan};
(2)~solvable groups of order prime to the number of roots of unity in $K$, by Neukirch's theorem~\cite{NeukirchSolvable};
and (3) groups with a generic extension over $K$, by Saltman's theorem \cite{Sal}.

Recent results by D\`ebes-Ghazi on the inverse Galois problem \cite{DG} and by Harari on weak approximation for homogeneous spaces \cite{HarariBulletinSMF}, suggest that for every finite group $G$ there exists a finite set $T:=T(G, K)$ of ``bad places'' of $K$ such that the Grunwald problem has an affirmative answer for $(G,K,S)$ for every set $S$ that is disjoint from $T$. In fact, the existence of such a set $T$ is implied by a conjecture of Colliot-Th\'el\`ene on the Brauer-Manin obstruction for rationally connected varieties. The connection is done by considering the following more general version of the Grunwald problem for finite $K$-groups $G$.

\bigskip
\noindent {\bf Approximation property.} A $K$-group  $G$ has \emph{(weak) approximation} in a set $S$ of places of $K$ if the natural restriction map
\[\H^1(K,G)\to\prod_{v\in S}\H^1(K_v,G),\]
is surjective. We shall say $G$ has {\it approximation away from $T$}, if $G$ has approximation in $S$ for every finite set $S$ of places of $K$ that is disjoint from $T$.
\bigskip

The approximation property for every finite $S\subset\Omega_K$ is equivalent to weak approximation for certain homogeneous spaces (hence its name), see Section \ref{sec:hom spaces}. Moreover, for constant groups $G$, it is equivalent to a positive answer to the Grunwald problem for $(G,K,S)$, see Section \ref{sec:cohomology}.

The existence of a finite set $T$ of ``bad places'' away from which the approximation property  holds is expected to hold for arbitrary finite $K$-groups, and is strongly related with arithmetic properties of homogeneous spaces.
Moreover, the existence of such a set $T$ for $K$-groups has been proved over arbitrary number fields $K$ for: (1)~abelian $K$-groups $G$ by Wang, cf.~\cite{Wan}; (2) iterated semidirect products $G= A_1\rtimes (A_2\rtimes \cdots \rtimes A_r)$ of abelian $K$-groups by Harari, cf.~\cite{HarariBulletinSMF};
and (3)~solvable $K$-groups $G$ of order prime to the number of roots of unity in an extension of $K$ splitting $G$ by the second author (here $T=\emptyset$, cf.~\cite{GLA-AFPHEH}).

Ever since Wang's work on the subject, triples $(G, K, S)$ having a negative answer to the Grunwald problem (and thus not having the approximation property) are known to exist. However, it is unclear what the set $T=T(G,K)$ should be. In fact, there is no explicit description of a set $T$  even for basic groups such as semidirect products of two abelian $p$-groups.

This paper gives both affirmative and negative answers to the Grunwald problem and the approximation property, suggesting that the set $T=T(G,K)$ can always be taken to be the union of the places of $K$ which divide the order of $G$ and those which  ramify in the minimal extension splitting $G$.

\subsection{Main results}
We give a precise description of the set $T$ of ``bad places'' for iterated semidirect products of abelian groups, complementing Harari's result \cite[Thm. 1]{HarariBulletinSMF}.

\begin{thm}\label{Main Thm Intro}
Let $K$ be a number field and $G$ be a finite $K$-group which is an iterated semidirect product $G= A_1\rtimes (A_2\rtimes \cdots \rtimes A_r)$ of abelian $K$-groups. Let $L/K$ be an extension splitting $G$, and $T$ the set of places that either divide the order of~$G$, or ramify in $L$. Then $G$ has approximation away from $T$.
\end{thm}

If $G$ is constant,  $L=K$ and hence the only condition on places $v\in S$ is to be prime to the order of $G$. The Grunwald problem has then an affirmative answer for such $(G, K ,S)$. The family of iterated semidirect products of abelian groups contains the dihedral groups, the Heisenberg groups of order $p^3$,  and the $p$-Sylow subgroups of the symmetric groups, cf.~\cite{Weir}, of $\mathrm{GL}_n(\mathbb{F}_q)$, and of other classical groups over $\mathbb{F}_q$ when $q$ is prime to $p$, cf.~\cite{Kal}.\\

In an opposite scenario, when $S$ consists of the primes of $K$ dividing the order of $G$ and $G$ is constant, we give the following examples in which the approximation property doesn't hold. 
Recall that a {\it bicyclic group} is either cyclic or a direct product of two cyclic groups.

\begin{thm}\label{thm couterexample intro}
Let $K$ be a number field and $G$  a finite abelian $p$-group that is not bicyclic. Assume that $G$ is a Galois group over some completion of $K$ (which a fortiori divides $p$) and let $S$ be the finite set of places of $K$ lying above $p$. Then there exists an abelian $G$-module $A$ (of order a power of $p$) such that if we consider $E := A \rtimes G$ as a constant $K$-group and $K$ contains sufficiently many roots of unity, the map
\[\H^1(K,E) \to \prod_{v \in S} \H^1(K_v, E),\]
is not surjective.
\end{thm}

In contrast to Wang's counterexamples, here the set $S$ consists of the primes of $K$ dividing any given rational prime $p$, as requested in \cite[Section~1.6]{DG} by D\`ebes-Ghazi. Examples of constant $p$-groups that do not admit the approximation property at the places dividing a given prime $p$  were also given over $K=\q(\mu_p)$ by the first author in \cite[\S6, Prop. 2]{Demarche}, where an  algebraic Brauer-Manin obstruction to the approximation property is given. However, in contrast with \cite[\S5]{HarariBulletinSMF}, \cite{Demarche} and \cite[\S5.2]{GLABrnral1}, which study algebraic Brauer-Manin obstructions, Theorem~\ref{thm couterexample intro} provides examples in which the approximation property doesn't hold and the algebraic Brauer-Manin obstruction vanishes, giving \emph{the first examples} of a transcendental Brauer-Manin obstruction to weak approximation on homogeneous spaces, see Example \ref{Example}.

\subsection{Tame problems}
The above results suggest that the answer to the Grunwald problem is affirmative away from a particular set of bad primes:

\bigskip

\noindent {\bf Tame approximation problem.} Does the approximation property
hold for every finite $K$-group $G$ and every finite set $S$ of places
of $K$ that are prime to its order and are unramified  in an extension splitting $G$?

\bigskip

Here we do not consider the extension $\bb{C}/\bb{R}$ to be ramified, hence the question about \emph{real approximation} (i.e. having the approximation property for the set $S$ of archimedean places), asked years ago by Borovoi, is included in this last one.

Negative answers to these questions are unknown, and a complete affirmative answer is currently out of reach, as it implies a solution to the inverse Galois problem (see \cite[\S4]{HarariBulletinSMF}). In view of Theorem \ref{Main Thm Intro} and its proof, it seems reasonable to conjecture that the answer is affirmative for all solvable groups.

\subsubsection*{Acknowledgements}
The authors would like to thank David Harari for being at the origin of this collaboration. The third author would like to thank Pierre D\`ebes, David Saltman, Nguy\^e\~n Duy T\^an and Jack Sonn for helpful discussions.

The first author acknowledges the support of the French Agence Nationale de la Recherche (ANR) under reference ANR-12-BL01-0005. The second author's work was partially supported by the Fondation Math\'ematique Jacques Hadamard through the grant N\textsuperscript{o} ANR-10-CAMP-0151-02 in the ``Programme des Investissements d'Avenir''. The third author's work was supported by the National Science Foundation under Award No. DMS-1303990.

\section{Preliminaries}

\subsection{Global and local Galois groups}\label{section notations}
All throughout this text, $K$ denotes a number field, $\Omega_K$ is the set of places of $K$ and, for $v\in\Omega_K$, $K_v$ denotes the completion of $K$ in $v$. For any Galois extension $L/K$ we denote by $\gal(L/K)$ the corresponding Galois group, and by $\gal(K)$ the absolute Galois group.
Throughout the text, we fix an embedding of  $\gal(K_v)$ in $\gal(K)$ and identify it with its image for each $v\in\Omega_K$.

For a finite place  $v\in\Omega_K$, let $K_v^{\tr}$ be the maximal tamely ramified extension of $K_v$, let $\Gamma_v:=\gal(K_v^\tr/K_v)$ be the tame Galois group and  $W_v:=\gal(K_v^\tr)$ the wild ramification group. Recall that the tame inertia group $T_v:=\gal(K_v^\tr/K_v^{\un})$ is a  procyclic normal subgroup of $\Gamma_v$ of order prime to $v$; Let $\tau_v$ be a generator of $T_v$. The quotient $\Gamma_v/T_v\cong \gal(K_v^\un/K_v)\cong \hat{\mathbb Z}$ is generated by a lift $\overline \sigma_v$ of the Frobenius automorphism.

We let $\sigma_v\in \Gamma_v$ be a preimage $\overline \sigma_v$. Note that $\sigma_v$ is defined up to conjugation and that its action by conjugation on $T_v$ is equivalent to its action on roots of unity, that is, $\sigma_v\tau_v\sigma_v^{-1}=\tau_v^{q_v}$ where $q_v$ is the cardinality of the residue field of $K_v$, see \cite[\S 7.5]{NSW}.

For  archimedean  $v$, let $\sigma_v$ be the generator of $\gal(K_v)$ and  put $\tau_v=1$, so that $\Gamma_v:=\gal(K_v)$ and $T_v=\langle\tau_v\rangle=\{1\}$.\\

\subsection{$K$-groups}
A finite $K$-group $G$  is a finite group scheme over $K$. Since $K$ is of characteristic 0, there is an equivalence of categories between finite group schemes over $K$ and finite $\gal(K)$-groups. Identifying the two, we shall write $G$ for the set of its geometric points. An extension $L/K$ is said to \emph{split} $G$ if $G\times_K L$ is a constant $L$-group, i.e. if the Galois group $\gal(L)\subset\gal(K)$ acts trivially on~$G$. We also need the following notion of ``bad places'' for such $G$:

\begin{defi}\label{definition bad places}
Let $K$ be a number field, $G$ be a finite $K$-group and $L/K$ the minimal extension splitting $G$. We define the set of ``bad places'' $\bad_G\subset\Omega_K$ as the union of the set of places dividing the order of $G$ and the places ramified in~$L/K$.
\end{defi}

Note that the minimal extension splitting $G$ always exists. Indeed, the action of $\gal(K)$ on $G$ is a (continuous) morphism $\gal(K)\to\aut(G)$ and the minimal extension splitting $G$ is given by the kernel of this morphism. 

\subsection{Cohomology} \label{sec:cohomology}
Recall \cite[I, \S5]{SerreCohGal} that for a field $K$ and  a $K$-group~$G$, the  set $\H^1(K,G)$ is defined as the quotient of the set of 1-cocycles (also called \emph{crossed homomorphisms})
\[Z^1(K,G):=\{a:\gal(K)\to G\,\text{ continuous}\, |\,a_{\sigma\tau}=a_\sigma\asd{\sigma}{}{}{\tau}{a},\,\,\forall\, \sigma,\tau\in\gal(K)\},\]
by the equivalence
\[a\sim b\,\Leftrightarrow\,\exists\, g\in G\text{ such that } a_\sigma=gb_\sigma\asd{\sigma}{}{-1}{}{g},\,\,\forall\,\sigma\in\gal(K).\]
Hence for a  \emph{constant} $K$-group $G$, the set $\H^1(K,G)$ is the set of continuous group homomorphisms $\gal(K)\to G$ modulo conjugation in $G$. In particular, the Grunwald problem for $(G,K,S)$ is equivalent to the approximation property for $G$ constant.
If $L/K$ is an extension splitting $G$, then $\gal(L/K)$ acts on $G$ and one may similarly define the set $\H^1(L/K,G)$, which embeds into $H^1(K,G)$ by inflation.

A class $\alpha\in \H^1(K,G)$ is denoted by $[a]$ if it corresponds to the class of the cocycle $a$. For $\alpha\in \H^1(K,G)$, we denote by $\alpha_v$ its image under the restriction map $\res_v:\H^1(K,G)\to \H^1(K_v,G)$.

\subsection{Twisting}\label{section twisting}
We recall briefly the basic properties of twisting. For further details see \cite[\S I.5]{SerreCohGal}.
Let $\Gamma$ be a profinite group and $G$ be a discrete $\Gamma$-group. Assume that $G$ and $\Gamma$ act both on the left on some object $X$ in a compatible way, that is, $\asd{\sigma}{}{}{}{(}g\cdot x)=\asd{\sigma}{}{}{}{g}\cdot\asd{\sigma}{}{}{}{x}$ for $\sigma\in\Gamma$, $g\in G$, $x\in X$. Given a cocycle $a\in Z^1(\Gamma,G)$,  define a new action of $\Gamma$ on $X$ by \emph{twisting} the first action as follows:
\[\asd{\sigma*}{}{}{}{x}:=a_\sigma\cdot\asd{\sigma}{}{}{}{x}.\]
Such a twisted object, denoted by $\asd{}{a}{}{}{X}$,  still has an action of $G$. However, the latter is not necessarily  compatible with the action of $\Gamma$. To fix this, one twists also the action of $\Gamma$ on $G$. To do so it suffices to view $G$ as acting on itself by conjugation, and repeat the same construction, so that
\[\asd{\sigma*}{}{}{}{g}:=a_\sigma\asd{\sigma}{}{}{}{g}a_\sigma^{-1}.\]
This twist of $G$, denoted by $\asd{}{a}{}{}{G}$, is a $\Gamma$-group which acts on $\asd{}{a}{}{}{X}$ compatibly with the $\Gamma$-action.

This construction takes principal homogeneous spaces under $G$ to principal homogeneous spaces under $\asd{}{a}{}{}{G}$. More precisely, one has \cite[I.5.3, Prop. 35bis]{SerreCohGal}:

\begin{pro}\label{prop twisting}
Let $a\in Z^1(\Gamma,G)$ and let $G'=\asd{}{a}{}{}{G}$. The map
\[t_a:Z^1(\Gamma,G')\to Z^1(\Gamma,G):a'\mapsto (\sigma\mapsto a'_\sigma a_\sigma),\]
is a bijection which moreover induces a bijection
\[\tau_a:\H^1(\Gamma,G')\to \H^1(\Gamma,G),\]
sending the trivial element of $\H^1(\Gamma,G')$ to the class of $a$ in $\H^1(\Gamma,G)$.
\end{pro}

An exact sequence of $\Gamma$-groups such as $1\to H\to E\to G\to 1$ gives rise to an exact sequence of pointed sets, cf.~\cite[I.5.5, Prop. 38]{SerreCohGal}:
\begin{equation}\label{equ:H1}
\H^1(\Gamma,H)\to \H^1(\Gamma,E)\to \H^1(\Gamma,G),
\end{equation}
which means that the elements in $\H^1(\Gamma,E)$ falling onto the trivial element of $\H^1(\Gamma,G)$ are precisely those coming from $\H^1(\Gamma,H)$.

Let $e\in Z^1(\Gamma,E)$ and let $g$ be its image in $Z^1(\Gamma,G)$. To study the fiber of the class $[e]\in \H^1(\Gamma,E)$, that is, the fiber above $[g]\in \H^1(\Gamma,G)$, one uses twisting. Since $E$ acts on $E$, on $G$ and on $H$ by conjugation, one may twist all three groups by $e$, getting an exact sequence $1\to \asd{}{e}{}{}{H}\to \asd{}{e}{}{}{E}\to \asd{}{e}{}{}{G}\to 1$ and the following commutative diagram of pointed sets with exact sequences:
\[\xymatrix{
\H^1(\Gamma,H) \ar[r] & \H^1(\Gamma,E) \ar[r] \ar[d]^{\tau_e^{-1}} & \H^1(\Gamma,G) \ar[d]^{\tau_g^{-1}} \\
\H^1(\Gamma,\asd{}{e}{}{}{H}) \ar[r] & \H^1(\Gamma,\asd{}{e}{}{}{E}) \ar[r] & \H^1(\Gamma,\asd{}{g}{}{}{G}).
}\]
Note that the twisted form of $G$ is denoted by $\asd{}{g}{}{}{G}$ instead of $\asd{}{e}{}{}{G}$, since $E$ acts on $G$ via its own image in $G$. Since the lower row is exact and since the $\tau$'s are bijections sending $[g]$ and $[e]$ to the trivial elements, we now know that the fiber over $[g]$ is in bijection with the image of the set $\H^1(\Gamma,\asd{}{e}{}{}{H})$ in $\H^1(\Gamma,\asd{}{e}{}{}{E})$.
Note also that in general there is no vertical arrow at the level of $H$ (as Proposition \ref{prop twisting} applies only to twists by a group acting on itself by conjugation).

Two elements in $\H^1(\Gamma,H)$ are mapped under \eqref{equ:H1} to the same element of $\H^1(\Gamma,E)$ if and only if they lie in the same orbit of $\coker(\H^0(\Gamma,E)\ra \H^0(\Gamma, G))$, \cite[I.5.5, Prop. 39]{SerreCohGal}. Here an element $g\in \H^0(\Gamma,G) = G^{\Gamma}$ acts on $\H^1(\Gamma, H)$ by sending a class $\alpha\in\H^1(\Gamma,H)$ to
\begin{equation}\label{equ:H0-act}
(g\cdot \alpha)_\sigma = e \alpha_\sigma \asd{\sigma}{}{}{}{e^{-1}}
\end{equation}
where $e\in E$ is a preimage of $g$.

\subsection{Homogeneous spaces}\label{sec:hom spaces} Let us now recall the notion of weak approximation for $K$-varieties. Let $X$ then be a (smooth, geometrically integral) $K$-variety such that $X(K)\neq \emptyset$.

\begin{defi}
Let $S\subset\Omega_K$ be a finite set of places of $K$. We say that $X$ has \emph{approximation in $S$} if $X(K)$ is dense in the product $\prod_{S} X(K_v)$.

We say that $X$ has \emph{weak approximation away from $T\subset\Omega_K$} if $X$ has approximation in $S$ for all finite $S\subset\Omega_K\smallsetminus T$. Equivalently, one can demand $X(K)$ to be dense in the product $\prod_{\Omega_K\smallsetminus T}X(K_v)$. The set $T$ is usually refered to as the set of ``bad places''.

We say that $X$ has \emph{weak approximation} if one can take $T=\emptyset$.
\end{defi}

Let us  briefly recall the Brauer-Manin obstruction to weak approximation, cf.~\cite[\S 5.1]{Skor} for details. For $X$ a (smooth, geometrically integral) $K$-variety, consider its unramified Brauer group $\brun X$, as well as the subgroup of its ``algebraic'' elements $\brunal X:=\ker[\brun X\to\brun(X\times_K \bar K)]$. Denote, for $v\in\Omega_K$, $\inv_v:\br K_v \to \q/\z$ the Hasse invariant map. Finally, denote by $X(K_\Omega)$ the product $\prod_{\Omega_K}X(K_v)$ in which $X(K)$ embeds diagonally. Then one can define the set $X(K_\Omega)^\brun$ (resp. $X(K_\Omega)^{\brunal}$) as the subset of all families of points $(P_v)_{v\in\Omega_K}\in X(K_\Omega)$ such that
\[\sum_{v\in\Omega_K}\inv_v(\alpha(P_v))=0,\text{ for all } \alpha\in\brun X\text{ (resp. } \alpha \in\brunal X),\]
where $\alpha(P_v)\in\br K_v$ denotes the evaluation of $\alpha$ at the point $P_v$ and the sum, which a priori is infinite, is actually finite for elements $\alpha\in\brun X$. The following inclusions then hold:
\[\overline{X(K)}\subseteq X(K_\Omega)^{\brun}\subseteq X(K_\Omega)^{\brunal}\subseteq X(K_\Omega).\]
If $X(K_\Omega)^{\brun}\neq X(K_\Omega)$ (resp. $X(K_\Omega)^{\brunal}\neq X(K_\Omega)$) one says that there is a Brauer-Manin obstruction (resp. an algebraic Brauer-Manin obstruction) to weak approximation. A Brauer-Manin obstruction that is not algebraic is called transcendental.\\

A conjecture by Colliot-Th\'el\`ene \cite[Introduction]{CT} says that the Brauer-Manin obstruction to weak approximation should be the only obstruction for rationally connected varieties, that is, $\overline{X(K)}= X(K_\Omega)^{\brun}$. Now, given a finite $K$-group~$G$, one can always embed it into $\sln$ for some $n$ and consider the homogeneous space $X=\sln/G$, which is unirational (since $\sln$ is itself a rational variety) and hence rationally connected. In \cite[Thm. 1]{HarariBulletinSMF}, Harari proved that if $G$ is an iterated semidirect product of abelian groups, then we do have $\overline{X(K)}=X(K_\Omega)^{\brun}$. The fact that this theorem implies the approximation property away from a finite set $T$ of places as we claimed in \S \ref{sec:intro} follows from the finiteness of $\brun X/\br K$, cf.~\cite[Prop.~5.1(iii)]{ColliotBrnr}, and from \cite[\S 1.2]{HarariBulletinSMF} or  \cite[\S 1]{GLA-AFPHEH}:

\begin{pro}\label{prop WA and H1}
Let $G$ a be $K$-group embedded into $\sln$ and put $X:=\sln/G$. Let $S\subset\Omega_K$ be a finite set of places of $K$. Then $X$ has approximation in $S$ if and only if the natural map
\[\H^1(K,G)\to\prod_{v\in S}\H^1(K_v,G),\]
is surjective.
\end{pro}

Note that this result proves in particular that approximation properties for varieties such as $X=\sln/G$ depend only on $G$, i.e. they are independent of the embedding of $G$ and even on the dimension of $X$ since there is no condition on $n$, justifying the definition of the approximation property in \S \ref{sec:intro}.

The tame approximation problem is thus equivalent to determining whether $G$ (or $X=\sln/G$) has weak approximation away from $T=\bad_G$. Note that a positive answer to this question does not necessarily imply a positive answer to Colliot-Th\'el\`ene's conjecture for $X=\sln/G$. Conversely, if the conjecture were true, we would only get the existence of a finite set $T_\br$ of bad places. However, not enough is known on the unramified Brauer group of $X$ in order to show that $T_\br=\bad_G$.

\subsection{Poitou-Tate}\label{section Poitou-Tate}
We  next recall the obstruction to weak approximation for finite abelian $K$-groups.
Let $A$ be a finite abelian $K$-group and let $\hat A=\hom(A,\gm)$ be its Cartier dual, which is also a finite abelian $K$-group.

Let $K_v^\un$ denote the  maximal unramified extension and $K_v^\tr$ its maximal tamely ramified extension.
Recall  \cite[II, \S6]{SerreCohGal} that for every finite place $v$ such that $\gal(K_v^\un)$ acts trivially on $A$, the unramified cohomology $\H^1_\un(K_v,A)$ is defined as the image of the group $\H^1(K_v^\un/K_v,A)$ under inflation to $\H^1(K_v,A)$. One can consider then the restricted product $\rprod_{\Omega_K}\H^1(K_v,A)$ with respect to these subgroups. It is well known that the product of the restriction maps $\res_v$ sends $\H^1(K,A)$ into $\rprod_{\Omega_K}\H^1(K_v,A)$. A classical result by Poitou and Tate \cite[II.6.3]{SerreCohGal} gives a perfect pairing
\begin{equation}\label{eq Poitou-Tate}
\rprod_{\Omega_K}\H^1(K_v,A)\times \rprod_{\Omega_K}\H^1(K_v,\hat A)\to\q/\z,
\end{equation}
defined via local pairings and such that the image of $\H^1(K,A)$ is the orthogonal complement of the image of $\H^1(K,\hat A)$ for this pairing. A classic consequence of these is the following well-known proposition  \cite[Lem. 9.2.2]{NSW}.

For a subset $S$ of $\Omega_K$, let $\Sha^1_S(K,G):=\{\alpha\in \H^1(K,G)\,|\,\alpha_v=1,\,\forall\,\, v\not\in S\}$, and let $\Sha^1(K,G):=\Sha^1_\emptyset(K,G)$.

\begin{pro}\label{Poitou-Tate lemma}
Let $A$ be a finite abelian $K$-group and $S\subset\Omega_K$ be a finite set of places. Then there is a pairing
\[\prod_{v\in S}\H^1(K_v,A)\times\Sha^1_S(K,\hat A)\to\q/\z,\]
such that its right kernel is $\Sha^1(K,\hat A)$ and its left kernel is the image of the restriction map
\[\H^1(K,A)\to \prod_{v\in S}\H^1(K_v,A)\]
In particular, $A$ has approximation in $S$ if and only if $\Sha^1_S(K,\hat A)=\Sha^1(K,\hat A)$.
\end{pro}
We shall use the following property of $\Sha^1_S(K,\cdot)$ to descend to finite extensions:

\begin{lem}\label{Lemma on Sha}
Let $A$ be a finite abelian $K$-group. Let $L/K$ be a Galois extension splitting $A$. Then for any finite $S\subset\Omega_K$, $\Sha^1_S(K,A)$ is contained in the image of $\H^1(L/K,A)$ by the inflation map.
\end{lem}

\begin{proof}
Since $\gal(L)$ acts trivially on $A$, we have the following commutative diagram with exact rows:
\[\xymatrix{
1\ar[r] & \H^1(L/K,A) \ar[r]^{\inf} & \H^1(K,A) \ar[d] \ar[r]^\res & \H^1(L,A) \ar[d] \\
&& \displaystyle{\prod_{v\in \Omega_K} \H^1(K_v,A)} \ar[r]^\res & \displaystyle{\prod_{w\in \Omega_{L}} \H^1(L_w,A)}.
}\]
Since $L$ splits $A$, the $\H^1$'s on the right hand side of the diagram are actually groups of homomorphisms. Restricting a class $\alpha\in\Sha^1_S(K,A)$  to $\H^1(L,A)$, we get a homomorphism $\gal(L)\to A$ that is trivial  everywhere locally except for the finitely many $w\in\Omega_{L}$ that lie above $S\subset\Omega_K$. By Chebotarev's density theorem it is the trivial morphism. Thus by exactness, $\alpha$ comes from $\H^1(L/K,A)$.
\end{proof}

\section{Reocurrence}\label{section reocurrence}
The main idea in proving Theorem \ref{Main Thm Intro} is to show that, given a finite $K$-group~$G$ of order $n$ and a finite place $v\not\in\bad_G$, there are infinitely many other places $w$ such that $\H^1(K_v,G)\cong \H^1(K_w,G)$. In the particular case of a constant group, this infinite set of places will only depend on $v$ and $n$. Hence, for a place $v$ prime to $n$, a group $G$ of order $n$ will appear as a Galois group over $K_v$ if and only if it appears over $K_w$.

Let us start by proving that, for finite places not in $\bad_G$, the set $\H^1(K_v,G)$ depends only on the tamely ramified part $\Gamma_v = \gal(K_v^\tr/K_v)$. Recall that $\Gamma_v$ is generated by a lift $\sigma_v$ of the Frobenius and a generator of the inertia group $\tau_v$, as in \S\ref{section notations}. For every place $v$, we fix a choice of $\sigma_v$ and $\tau_v$ and keep it all throughout the text.

\begin{lem}\label{lemma local H1s}
Let $G$ be a finite $K$-group of order $n$ and let $v\in\Omega_K$ be a place outside $\bad_G$. Denote by $T_v^n$ the subgroup of $\Gamma_v$ generated by $\tau_v^n$. Then there is an isomorphism $\H^1(K_v,G)\cong \H^1(\Gamma_v/T_v^n,G)$ given by inflation.
\end{lem}

\begin{proof}
The statement is trivial for archimedean $v$, so we may assume that $v$ is finite. Since the wild ramification subgroup $W_v\leq \Gal(K_v)$ acts trivially on $G$,  the inflation-restriction sequence gives:
\[1\to \H^1(\Gamma_v,G)\xrightarrow{\inf} \H^1(K_v,G)\xrightarrow{\res} \H^1(W_v,G),\]
and $\H^1(W_v,G)$ is the set of morphisms $W_v\to G$ up to conjugation. Since $W_v$ is a pro-$p$ group with $p$ not dividing $n$, there are no such nontrivial morphisms and hence $\H^1(K_v,G)\cong \H^1(\Gamma_v,G)$.

Consider now a class $\alpha\in \H^1(\Gamma_v,G)$. Since $v$ is unramified in the minimal extension splitting $G$, $T_v$ also acts trivially on $G$ and hence $\alpha$ restricts to a morphism (up to conjugation) from $T_v$ to $G$. It is then evident that this morphism is trivial over $T_v^n$. The same inflation-restriction argument then gives $$\H^1(K_v,G)\cong \H^1(\Gamma_v/T_v^n,G).$$
\end{proof}

The following reocurrence result generalizes the statement given in the beginning of the section.

\begin{pro}\label{prop reocurrence}
Let $G$ be a finite $K$-group of order $n$, $L/K$ a Galois extension splitting~$G$ and let $v\in\Omega_K$ be either an archimedean place or a finite place which is unramified in $L$ and does not divide $n$ (in particular, $v\not\in \bad_G$). Then there exist infinitely many finite places $w\not\in \bad_G$ for which:
\begin{enumerate}
\item  the decomposition groups of $v$ and $w$ in $\gal(L/K)$ are conjugate;
\item  there is an epimorphism $\phi:\Gamma_w/T_w^n\twoheadrightarrow \Gamma_{v}/T_{v}^n$ given by  $\phi(\sigma_w)=\sigma_v, \phi(\tau_w)=~\tau_v$, which is moreover an isomorphism if $v$ is finite.
\end{enumerate}
The epimorphism $\phi$ induces a monomorphism $\phi^*:\H^1(K_{v},G)\hookrightarrow \H^1(K_{w},G),$
which is moreover an isomorphism if $v$ is finite.
\end{pro}

\begin{proof}
By Chebotarev's density theorem applied to the Galois extension $L(\mu_{n})/K$ (which is unramified in $v$) there are infinitely many places $w\not\in\bad_G$ for which the image of the decomposition group in $\gal(L(\mu_n)/K)$ is generated by the image of~$\sigma_w$ and is conjugate to the image of $\sigma_v$. In particular, we get that these groups are conjugate in $\gal(L/K)$.

If $v$ is finite, the images of $\sigma_v$ and $\sigma_{w}$ coincide in the quotient $\gal(K(\mu_n)/K)$ and hence their residue degrees $q_v,q_w$, respectively, are congruent mod $n$. Moreover, recall that the group $\Gamma_v/T_v^n$ has the following presentation:
\[ \langle \sigma_v,\tau_v\,|\,\sigma_v\tau_v\sigma_v^{-1}=\tau_v^{q_v},\tau_v^n=1\rangle .\]
Since $q_w\equiv q_{v}\mod n$, we get an isomorphism $\phi:\Gamma_w/T_w^n\to \Gamma_{v}/T_{v}^n$ by setting $\phi(\sigma_w)=\sigma_{v}$ and $\phi(\tau_w)=\tau_{v}$. If $v$ is archimedean, the same definition gives an epimorphism $\phi:\Gamma_w/T_w^n\to \gal(K_v)$ since $\sigma_v^2=\tau_v=1$.\\

Let $\rho\in\gal(K)$ be an element for which $\sigma_{w}$ coincides with $\rho\sigma_{v}\rho^{-1}$ in $\gal(L/K)$. Denote by $\asd{}{\rho}{}{}{G}$ the $K_v$-group where $\sigma\in\gal(K_v)$ acts on $G$ by $g\mapsto\asd{\rho\sigma\rho^{-1}}{}{}{}{g}$ and note that this action coincides with that of $\gal(K_w)$ on $G$. Consider the  maps
\[\H^1(\Gamma_{v}/T_{v}^n,G)\xrightarrow{\rho_*} \H^1(\Gamma_v/T_v^n,\asd{}{\rho}{}{}{G})\xrightarrow{\phi^*} \H^1(\Gamma_w/T_w^n,G),\]
where $\rho_*$ is a canonical isomorphism defined at the level of cocycles by the identity on $\Gamma_v/T_v^n$ and by sending $g\in G$ to $\asd{\rho}{}{}{}{g}$. We abusively denote by $\phi^*$ the whole composition. Then evidently $\phi^*$ is an isomorphism if $v$ is finite and injective as an inflation morphism if $v$ is archimedean. Recall finally that by Lemma \ref{lemma local H1s}, we have  $\H^1(K_v,G)\cong \H^1(\Gamma_v/T_v^n,G)$ for finite $v$ and $\Gamma_v/T_v^n=\gal(K_v)$ for archimedean $v$, hence the result.
\end{proof}

We finally give a particular application necessary for the proof of Theorem \ref{Main Thm Intro}:

\begin{lem}\label{reocurrence lemma}
Let $G$, $L/K$, $v$, $w$, $\phi$ and $\phi^*$ be as given by Proposition \ref{prop reocurrence}. Let $\alpha\in \H^1(K_{v},G)$ and assume that there exists a class $\beta=[b]\in \H^1(K,G)$ such that $\beta_{v}=\alpha$, $\beta_{w}=\phi^*\alpha$ and such that the group morphism $\gal(L)\to G$ obtained by restriction of $b$ to $\gal(L)$ is surjective. Denote by $L'$ the extension defined by the kernel of this morphism. Then the decomposition groups of $v$ and $w$ in $\gal(L'/K)$ are conjugate.
\end{lem}
We  first note that $L'/K$ is indeed Galois:
\begin{lem}\label{lemma Galois extension}
Let $G$ be a $K$-group and $L/K$ be a Galois extension splitting $G$. Let $b\in Z^1(K,G)$ and $L'/L$ be the field fixed by the kernel of the restriction of $b$ to $\gal(L)$. Then $L'$ is Galois over $K$.
\end{lem}

\begin{proof}
Let $\tau\in\gal(L')$ and $\sigma\in\gal(K)$. Then
\[b_{\sigma\tau\sigma^{-1}}=b_\sigma\asd{\sigma}{}{}{\tau}{b}\asd{\sigma\tau\sigma^{-1}}{}{-1}{\sigma}{b}.\]
Now, since $L/K$ is Galois and $\tau\in\gal(L')\subseteq \gal(L)$, we get $\sigma\tau\sigma^{-1}\in\gal(L)$, hence it acts trivially on $b_\sigma$. As $b_\tau=1$ by definition of $L'$, we get  $b_{\sigma\tau\sigma^{-1}}=1$, which proves that $\sigma\tau\sigma^{-1}\in\gal(L')$ and thus $L'/K$ is Galois.
\end{proof}

\begin{proof}[Proof of Lemma \ref{reocurrence lemma}]
We first show that the decomposition groups of $v$ and $w$ in $\gal(L'/K)$ are quotients of $\Gamma_{v}/T_{v}^n$ and $\Gamma_w/T_w^n$. Indeed, this holds for archimedean $v$ since $\Gal(K_v)=\Gamma_v/T_v^n$. For finite $v$ and $w$, $L'/K$ is tamely ramified since $L/K$ is unramified and $L'/L$ is of degree $n$ and hence prime to $v$ and $w$. Since the restriction of $b$ to $T_v$ and $T_{w}$ is a morphism, it must be trivial on $T_v^n$ and $T_{w}^n$, i.e. $T_v^n, T_{w}^n\subset\gal(L')$, proving the claim.

Denote then (abusively) by $\sigma_{v},\sigma_w,\tau_{v},\tau_w,$ the images  in $\gal(L'/K)$ of these elements. Recall that by the proof of Proposition \ref{prop reocurrence}, there exists an element $\rho\in\gal(K)$ such that $\sigma_{w}$ coincides with $\rho\sigma_{v}\rho^{-1}$ in $\gal(L/K)$. Also note that  $\rho\tau_v\rho^{-1}=~\tau_w=~1$ in $\gal(L/K)$. We claim that we may choose $\rho$ so that $b_\rho=1$. Note that $b_{\chi\rho}=b_\chi b_\rho$ for $\chi\in\gal(L)$. Since $b$ is surjective when restricted to $\gal(L)$, there exists $\chi\in\gal(L)$ such that $b_\chi=b_\rho^{-1}$ so that $b_{\chi\rho}=1$. Since $\chi\in\gal(L)$, we have that $\sigma_{w}$ coincides with $(\chi\rho)\sigma_{v}(\chi\rho)^{-1}$ in $\gal(L/K)$, proving the claim.

The isomorphism $\phi^*$ between $\H^1(K_v,G)$ and $\H^1(K_{w},G)$ is given at the level of cocycles by the morphisms $\rho_*:G\to G$ and $\phi:\Gamma_w/T_w^n\to \Gamma_{v}/T_{v}^n$, which satisfy
$\rho_*(g)=\asd{\rho}{}{}{}{g}$, $\phi(\sigma_w)=\sigma_v$, and $\phi(\tau_w) = \tau_v$. In particular, if $\alpha=[a]$ for $a\in Z^1(\Gamma_{v}/T_{v}^n,G)$, then $\phi^*\alpha=[a']$ with $a'_{\sigma_{w}}=\asd{\rho}{}{}{\sigma_v}{a}$ and $a'_{\tau_w}=\asd{\rho}{}{}{\tau_{v}}{a}$.

Now, since $\beta_{v}=\alpha$ and $\beta_{w}=\phi^*\alpha$, we know that there exist $g,g'\in G$ such that
\begin{align*}
\asd{\rho}{}{}{}{(}gb_{\sigma_{v}}\asd{\sigma_v}{}{-1}{}{g})=\asd{\rho}{}{}{\sigma_{v}}{a}&=a'_{\sigma_{w}}=g' b_{\sigma_{w}}\asd{\sigma_w}{}{-1}{}{g'},\\
\asd{\rho}{}{}{}{(}gb_{\tau_{v}}g^{-1})=\asd{\rho}{}{}{\tau_{v}}{a}&=a'_{\tau_{w}}=g' b_{\tau_{w}}{g'}^{-1}.
\end{align*}
Since $b$ restricted to $\gal(L'/L)$ is an isomorphism by hypothesis, there are unique $\chi,\chi'\in\gal(L'/L)$ such that $b_\chi=g$, $b_{\chi'}=g'$, so that
\begin{align}\label{equ:test}
\begin{array}{rl}
\asd{\rho}{}{}{\chi}{b} \asd{\rho}{}{}{\sigma_{v}}{b}\asd{\rho\sigma_v}{}{-1}{\chi}{b} &=b_{\chi'} b_{\sigma_{w}}\asd{\sigma_w}{}{-1}{\chi'}{b},\\
\asd{\rho}{}{}{\chi}{b} \asd{\rho}{}{}{\tau_{v}}{b} \asd{\rho}{}{-1}{\chi}{b} &=b_{\chi'} b_{\tau_{w}} {b}_{\chi'}^{-1}.
\end{array}
\end{align}
Note now that since $b_\rho=1$, we have $b_{\rho\eta\rho^{-1}}=\asd{\rho}{}{}{\eta}{b}$ for every $\eta\in\gal(L'/L)$, so in particular for $\eta=\chi,\tau_v$, and $\chi^{-1}$. Noting that $\rho\Gal(L'/L)\rho^{-1}\subseteq \Gal(L'/L)$ since $L/K$ is Galois, that the restriction of $b$ to $\Gal(L'/L)$ is a homomorphism, and putting $\rho_0:={\chi'}^{-1}\rho\chi\in\gal(L'/K)$,  \eqref{equ:test} gives $b_{\rho_0\tau_v\rho_0^{-1}}=b_{\tau_w}$ and hence $\rho_0\tau_v\rho_0^{-1}=\tau_w$ in $\gal(L'/L)\subset\gal(L'/K)$.

Finally, we claim that $\rho_0\sigma_v\rho_0^{-1}=\sigma_w$ in $\gal(L'/K)$.
Since $\chi,\chi'\in\gal(L)$, we already know that $\sigma_w$ and $\rho_0\sigma_v\rho_0^{-1}$ coincide on $\gal(L/K)$ and hence act equally on $G$. Thus, on the one hand, \eqref{equ:test} gives
\[b_{\rho_0}\asd{\rho}{}{}{\sigma_v}{b}\asd{\sigma_w}{}{-1}{\rho_0}{b}=b_{\sigma_w}.\]
On the other hand, we have  $\rho_0\sigma_v\rho_0^{-1}=\psi\sigma_w$ for some $\psi\in\gal(L'/L)$. Applying $b$ to this equality another direct computation gives
\[b_{\rho_0}\asd{\rho}{}{}{\sigma_v}{b}\asd{\sigma_w}{}{-1}{\rho_0}{b}=b_\psi b_{\sigma_w}.\]
From the two last equalities we get that $b_\psi=1$.  Since $b$ is an isomorphism over $\gal(L'/L)$,  we get $\psi=1$, proving the claim. Thus conjugation by $\rho_0$ sends the decomposition group of $v$ in $L'/K$ to that of $w$.
\end{proof}

\section{Weak approximation away from $\bad_G$}\label{section Main Thm}
We now restate and prove Theorem \ref{Main Thm Intro} in the language of approximation properties. Recall that, given a $K$-group $G$, $\bad_G$ denotes the set of bad places associated to it, cf.~Definition \ref{definition bad places}.

\begin{thm}\label{Main Thm}
Let $K$ be a number field and $G$ be a finite $K$-group with weak approximation away from $\bad_G$. Then every semidirect product $E= A\rtimes G$ with $A$ an abelian $K$-group has weak approximation away from $\bad_E$.
\end{thm}

\begin{rem}
Note that this result does imply Theorem \ref{Main Thm Intro}. Indeed, one can prove it by induction starting with the fact that the trivial group has weak approximation.
\end{rem}

\begin{proof}[Proof of Theorem \ref{Main Thm}]
Consider a finite set of places $S\subset\Omega_K$ not meeting $\bad_E$ and local classes $\beta_v\in \H^1(K_v,E)$ for $v\in S$. We will find a global class $\beta\in \H^1(K,E)$ mapping to the $\beta_v$'s in three steps:

\subsection*{Step 1:} Constructing a class $\beta'\in \H^1(K,E)$ with prescribed images in $\H^1(K_v,G)$ for $v$ in $S\cup S'$ for a suitable $S'$.

We have the following commutative diagram of pointed sets with exact rows:
\begin{equation}\label{big diagram}
\xymatrix{
 \H^1(K,A) \ar[d] \ar[r] & \H^1(K,E) \ar[d] \ar[r] & \H^1(K,G) \ar@/_1.5pc/[l] \ar[d]  \\
 \displaystyle{\prod_{v\in \Omega_K} \H^1(K_v,A)} \ar[r] & \displaystyle{\prod_{v\in \Omega_K} \H^1(K_v,E)} \ar[r] & \displaystyle{\prod_{v\in \Omega_K} \H^1(K_v,G)}, \ar@/_2pc/[l]
}
\end{equation}
where the arrows going left are set-theoretical sections induced by a fixed section~$G\to E$.

The $\beta_v$'s give us images $\gamma_v\in \H^1(K_v,G)$ for $v\in S$. Let $n=|E|$ and $L$ a Galois extension splitting $E$. Then, by Proposition \ref{prop reocurrence} applied to the extension $L(\mu_n)/K$ (which splits $G$), we know that for every $v\in S$ there exist places $v'\not\in S\cup \bad_E$ for which there are inclusions
\[\phi_v^*:\H^1(K_{v},G)\hookrightarrow \H^1(K_{v'},G).\]
Choose one such $v'$ for each $v\in S$ and denote by $S'$ the set of these places. We may assume that all the $v'$ are different since we have an infinite choice at each time. Define then $\gamma_{v'}:=\phi^*_v(\gamma_v)\in \H^1(K_{v'},G)$.\\

Since we've chosen these new places away from $\bad_E\supset \bad_G$, we know by hypothesis that there exists a global class $\gamma\in \H^1(K,G)$ mapping onto $\gamma_v$ for every $v\in~S\cup S'$. Moreover, we may assume that the restriction of a cocycle $c\in Z^1(K,G)$, representing $\gamma$, to $\gal(L(\mu_n))$ is  surjective: Indeed by Chebotarev's density theorem, there are infinitely many places $w\not\in S\cup S'\cup \bad_E$ totally split in $L(\mu_n)$, so in particular such that $G$ is constant over $K_w$; then, for each element $g\in G$ we may choose one such $w$ and an unramified class $\gamma_w\in \H^1(K_w,G)$ represented by a morphism sending $\sigma_w$ to $g$, so that all $w$'s are distinct; then finally, adding these local conditions to $S\cup S'$ forces $c$ to be surjective when restricted to $\gal(L(\mu_n))$.

Let $\beta'$ be the image of $\gamma$ in $\H^1(K,E)$ via the  section of diagram \eqref{big diagram}. Viewing $G$ as a subgroup of $E$, the class $\beta'$ is represented by the same cocycle $c$.

\subsection*{Step 2: Twisting.} We twist by $c$ to look for a class in $\H^1(K,E)$ satisfying the necessary prescribed local conditions in $S$.

Let us now twist diagram \eqref{big diagram} by the cocycle $c$ representing $\beta'$. This gives us the following new diagram with exact rows
\begin{equation}\label{twisted diagram}
\xymatrix{
\H^1(K,\asd{}{c}{}{}{A}) \ar[d] \ar[r] & \H^1(K,\asd{}{c}{}{}{E}) \ar[d] \ar[r] & \H^1(K,\asd{}{c}{}{}{G}) \ar[d] \\
\displaystyle{\prod_{v\in \Omega_K} \H^1(K_v,\asd{}{c}{}{}{A})} \ar[r] & \displaystyle{\prod_{v\in \Omega_K} \H^1(K_v,\asd{}{c}{}{}{E})} \ar[r] & \displaystyle{\prod_{v\in \Omega_K} \H^1(K_v,\asd{}{c}{}{}{G})}.
}
\end{equation}
For every cohomology class $\xi$ in a set of diagram \eqref{big diagram}, we denote by \asd{}{c}{}{}{\xi} its twisted image in the corresponding set of diagram \eqref{twisted diagram}.

Since the images of $\beta_v$ and $\beta'$ coincide on $\H^1(K_v,G)$ for $v\in S$ by construction and $\asd{}{c}{}{}{\beta'}$ is trivial by the very definition of twisting, we know by exactness that $\asd{}{c}{}{v}{\beta}\in \H^1(K_v,\asd{}{c}{}{}{E})$ comes from an element $\alpha_v\in \H^1(K_v,\asd{}{c}{}{}{A})$.

It suffices then to find a global class $\alpha\in \H^1(K,\asd{}{c}{}{}{A})$ mapping onto $\alpha_v$ for $v\in~S$ to conclude. Indeed, the image of $\alpha$ in $\H^1(K,\asd{}{c}{}{}{E})$ would then map onto the $\asd{}{c}{}{v}{\beta}$'s for $v\in S$ and hence its preimage in $\H^1(K,E)$ by the twisting morphism would map onto the $\beta_v$'s as desired. Thus, in order to conclude, it will suffice by Proposition~\ref{Poitou-Tate lemma} to prove that $\Sha^1_S(K,\asd{}{c}{}{}{\hat A})=\Sha^1(K,\asd{}{c}{}{}{\hat A})$.

\subsection*{Step 3: Proof of $\Sha^1_S(K,\asd{}{c}{}{}{\hat A})=\Sha^1(K,\asd{}{c}{}{}{\hat A})$.}
Consider a class $\alpha\in\Sha^1_S(K,\asd{}{c}{}{}{\hat A})$, i.e. a class in $\H^1(K,\asd{}{c}{}{}{\hat A})$ such that its image $\alpha_v\in \H^1(K_v,\asd{}{c}{}{}{\hat A})$ is trivial for every $v\not\in S$, so in particular for $v'\in S'$. We must show that $\alpha_v=0$ for $v\in S$ too.

Fix then $v\in S$ and its corresponding $v'\in S'$. Lemma \ref{reocurrence lemma} applies in this context to the local classes $\gamma_v$ and $\gamma_{v'}=\phi_v^*(\gamma_v)$, the global class $\gamma$ and the extension $L(\mu_n)/K$. By the lemma the decomposition groups of $v$ and $v'$ in $\gal(L'/K)$ are conjugate, where $L'$ is the extension of $L(\mu_n)$ given by the kernel of $c$ restricted to $\gal(L(\mu_n))$. Note moreover that $L'$ splits $\asd{}{c}{}{}{\hat A}$. Indeed, since $|A|$ divides $n$ we have $\asd{}{c}{}{}{\hat A}=\hom(\asd{}{c}{}{}{}{A},\mu_n)$ and since $L'$ contains $\mu_n$, then this amounts to $L'$ splitting $\asd{}{c}{}{}{A}$. The latter holds since the action of $\sigma\in\gal(K)$ on $a\in\asd{}{c}{}{}{A}$ is given by $c_\sigma\asd{\sigma}{}{}{}{a}c_\sigma^{-1}$, where $\asd{\sigma}{}{}{}{a}$ denotes the action of $\sigma$ on $a$ as an element of $A$. Since $c$ is trivial over $\gal(L')$ and $L'$ clearly splits $A$, we get then our claim. In particular, by Lemma \ref{Lemma on Sha}, we know that the whole group $\Sha^1_S(K,\asd{}{c}{}{}{\hat A})$ comes by inflation from $\H^1(L'/K,\asd{}{c}{}{}{\hat A})$.

Our element $\alpha\in\Sha^1_S(K,\asd{}{c}{}{}{\hat A})$ comes then from $\H^1(L'/K,\asd{}{c}{}{}{\hat A})$ and hence $\alpha_v$ comes from $\H^1({L'}^{(v)}/K_{v},\asd{}{c}{}{}{\hat A})$, where ${L'}^{(v)}$ denotes the (unique) extension of $K_v$ induced by $L'/K$. The same goes for $v'$. Since the decomposition groups of $v$ and $v'$ are conjugate in $\gal(L'/K)$, there is a canonical isomorphism
\[\H^1({L'}^{(v)}/K_{v},\asd{}{c}{}{}{\hat A})\xrightarrow{\sim} \H^1({L'}^{(v')}/K_{v'},\asd{}{c}{}{}{\hat A}),\]
which is compatible with the restrictions from $\H^1(L'/K,\asd{}{c}{}{}{\hat A})$. But $\alpha\in\Sha^1_S(K,\asd{}{c}{}{}{\hat A})$ and $v'\not\in S$, so that $\alpha_{v'}=0$ and hence $\alpha_v=0$ too. Since the same argument works for every $v\in S$, we deduce that $\alpha\in\Sha^1(K,\asd{}{c}{}{}{\hat A})$, proving the claim.
\end{proof}

\section{Counterexamples to weak approximation}\label{section counterexamples}
In this section we show that Theorem \ref{Main Thm} is sharp in the sense that one cannot expect to get approximation (or to solve the Grunwald problem) on the set of bad places.

Recall that a group is bicyclic if and only if it is cyclic or a direct product of two cyclic groups. For a finite group $G$ and a finite $G$-module $A$, we let
\[\Sha^1_\bic(G,A) :=\ker\left[\H^1(G,A)\to \prod_{H\in\bic(G)}\H^1(H,A)\right],\]
where $\bic(G)$ denotes the set of bicyclic subgroups of $G$.

\begin{thm}\label{thm counterexample}
Let $G$ be a finite abelian group, $K$ a number field, and $S$ a finite set of places of $K$. Let $A$ be a finite $G$-module and $E=A\rtimes G$ a constant $K$-group. Assume:
\begin{enumerate}
\item $K$ contains the $\exp(A)$-th roots of unity;
\item there exists $v_0 \in S$ such that $G$ is a Galois group over $K_{v_0}$;
\item $S$ contains all the places of $K$ dividing the order of~$G$;
\item $\Sha^1_\bic(G, \hat A) \neq 0$ (in particular $G$ is not bicyclic).
\end{enumerate}
Then the map
$\H^1(K,E) \to \prod_{v \in S} \H^1(K_v, E)$
is nonsurjective.
\end{thm}

\begin{rem}
Note that, given the structure of local Galois groups recalled in Section \ref{section notations}, Condition (4) implies that the place $v_0$ in Condition (2) must divide the order of~$G$.
\end{rem}

\begin{cor}\label{cor counterexample}
Assume in addition that $K$ contains the $\exp(E)$-th roots of unity, then there is a transcendental Brauer-Manin obstruction to weak approximation on $X := \sln/E$.
\end{cor}

\begin{proof}
Indeed, Theorem \ref{thm counterexample} and Proposition \ref{prop WA and H1} tell us that $X$ does not have approximation in $S$ and hence it doesn't have weak approximation. Now, by \cite[Theorem 1]{HarariBulletinSMF} the Brauer-Manin obstruction to weak approximation for such a variety is the only obstruction, whereas the formula given in \cite[Prop 5.9]{GLABrnral1} tells us that the algebraic part of $\brun X$ is trivial, hence the lack of weak approximation for $X$ cannot be explained by algebraic obstructions. The Brauer-Manin obstruction must then come from transcendental elements of $\brun X$.
\end{proof}

\begin{exe}\label{Example}
Let $p$ be a prime number and consider the group $G :=(\z/p\z)^3$. Let $A:=\hat I$, where $I$ is the augmentation ideal of $(\z/p^3\z)[G]$. There is a natural action of $G$ on $A$ and $E:= A\rtimes G$ is a group of order $p^{3 p^3}$. Define $K$ to be $\q(\zeta_{p^4})$ and $S$ to be the unique place of $K$ dividing $p$. By Lemma \ref{lemma examples} here below, we have $\Sha^1_\bic(G,\hat A)\neq 0$ and it is easy to see that all the other assumptions of Corollary~\ref{cor counterexample} are satisfied, hence the group $E$ provides the first explicit example of a transcendental Brauer-Manin obstruction to weak approximation on a homogeneous space.
\end{exe}

\begin{lem}\label{lemma examples}
Let $G$ be a finite group of order $n$ which is not bicyclic, let $R$ be the group ring $(\z/n\z)[G]$ and $I\lhd R$ the augmentation ideal. Then $\Sha^1_\bic(G,I)$ is nontrivial.
\end{lem}

\begin{proof}
For a $G$-module $A$, denote by $\hat \H^0(G,A)$  the Tate cohomology group $A^G/N_G(A)$ where $N_G(A):=\{\sum_{\sigma\in G} \asd{\sigma}{}{}{}{a} \, |\, a\in A\}$. Put $B:=R/I\cong\z/n\z$. Since $R$ is an induced $H$-module for any subgroup $H\leq G$, we have $\hat \H^0(H,B)\cong \H^1(H,I)$ and hence $\Sha_\bic^1(G,I)$ is isomorphic to
\[\Sha_\bic^0(G,B)=\ker\left[\hat \H^0(G,B)\to \prod_{H\in \bic(G)} \hat \H^0(H,B)\right].\]
Since $\hat \H^0(H,B)\cong \z/|H|\z$ for every $H \leq G$, and since $G\not\in\bic(G)$, we have:
\[\Sha^0_\bic(G,B)\cong \ker\left[\z/n\z \to \prod_{H\in\bic(G)}\z/|H|\z \right]\neq 0.\]
\end{proof}

\begin{proof}[Proof of Theorem \ref{thm counterexample}]
We divide the proof into five steps.

\subsection*{Step 1: Setup} We construct local classes in $\H^1(K_{v},G)$, the fibers of which will contain elements that cannot be approximated by a global element in $\H^1(K,E)$.

By Condition (2), there exists an epimorphism $c_{v_0} : \gal(K_{v_0}) \to~G$. Since $G$ is abelian and constant,  $\hom(\gal(K),G)=Z^1(K,G)=\H^1(K,G)$ and their equivalents for the $K_v$'s. One can then look at $c_{v_0}$ as an element of $\H^1(K_{v_0},G)$. We will omit these identifications from now on and abusively use the same letters for cocycles and cohomology classes for $G$. For $v\in S\smallsetminus\{v_0\}$, let $c_v := 0 \in \H^1(K_v, G)$.

Fix a section $s : G \to E$ of the following exact sequence
\begin{equation}\label{ses}
\xymatrix{
1 \ar[r] & A \ar[r] & E \ar[r]^{\pi} & G \ar[r] \ar@/_1pc/[l]_{s} & 1.
}
\end{equation}
Consider $s(c_{v}) \in Z^1(K_{v}, E)$, which we still denote abusively by $c_v$ by looking at $G$ as a subgroup of $E$ via $s$, and let $\beta_{v}:=[c_v]\in \H^1(K_v, E)$. If $(\beta_v)_{v \in S}$ is not contained in the image of the restriction map $\H^1(K, E) \to \prod_{S} \H^1(K_v, E)$, then we are done. So we can assume that there exists $\beta \in \H^1(K, E)$ lifting the $\beta_v$'s for all $v \in S$. Up to replacing $\beta$ by $s(\pi(\beta))$, one can assume that $\beta=s(c)$ for some $c\in \H^1(K,G)$ lifting the $c_v$'s for $v\in S$. Making once again an abuse of notation, we'll write $\beta=[c]$, viewing $c$ as an element of $Z^1(K,E)$.

\subsection*{Step 2: Twisting.} To study the fibers over $c_v,v\in S$, we twist the exact sequence \eqref{ses} by the cocycle $c$. We get the following twisted exact sequence:
\begin{equation*}
\xymatrix{
1 \ar[r] & \asd{}{c}{}{}{A} \ar[r] & \asd{}{c}{}{}{E} \ar[r] & \asd{}{c}{}{}{G} \ar[r] \ar@/_1pc/[l] & 1,
}
\end{equation*}
where we immediately remark that $\asd{}{c}{}{}{G}=G$ since $G$ is abelian. Moreover, since $c$ takes values in $G$, the section $s$ is still well defined on these twisted forms.

By definition of twisting, the action of $\gal(K)$ on $\asd{}{c}{}{}{A}$ is given by composition of the morphism $c:\gal(K)\to G$ with the morphism $G\to\aut(A)$ defining the action of $G$ on $A$.

\subsection*{Step 3: Defining unattainable classes in $\H^1(K_v,E)$. }
By Condition (1) of Theorem \ref{thm counterexample}, we know that $\hat A=\hom(A,\q/\z)$ and the same goes for its twisted version $\asd{}{c}{}{}{A}$. Since both $G$ and $\gal(K)$ act trivially on $\q/\z$,  the action of $\gal(K)$ on $\asd{}{c}{}{}{\hat A}$ corresponds via $c$ with the action of $G$ on $\hat A$, so that we get a pullback morphism $c^* : \H^1(G, \hat A) \to \H^1(K,\asd{}{c}{}{}{\hat A})$.

Note  that since $G$ is abelian, Condition (3) of Theorem \ref{thm counterexample} and the structure of local Galois groups imply that the image of $c:\gal(K) \to G$ restricted to $\gal(K_v)$ is a bicyclic subgroup, for all cocycles $c \in Z^1(K, G)$ and all places $v \not\in S$. Therefore, for any $c \in Z^1(K,G)$, the morphism $c^*$ restricts to a morphism
\[c^* : \Sha^1_\bic(G, \hat A) \to \Sha^1_S(K, \asd{}{c}{}{}{\hat A}).\]

Condition (4) tells us now that there exists $\gamma \neq 0 \in \Sha^1_\bic(G, \hat A)$. Since $c$ lifts $c_{v_0}$, and the latter was chosen to be  surjective,  $c$ must also be surjective. Hence the above map ${c}^*$  is an inflation map and  is therefore injective, so that ${c}^*(\gamma) \neq 0 \in \Sha^1_S(K,\asd{}{c}{}{}{\hat A})$.

In addition since $c_{v_0}$ is surjective,  one has by inflation that ${c}^*(\gamma)_{v_0}\neq 0$ in $\H^1(K_{v_0}, \asd{}{c}{}{}{\hat A})$, hence ${c}^*(\gamma) \not\in \Sha^1(K, \asd{}{c}{}{}{\hat A})$. By Proposition \ref{Poitou-Tate lemma}, there exists therefore $(\alpha_v)_{v \in S} \in \prod_{S} \H^1(K_v,\asd{}{c}{}{}{A})$ which is not in the image of $\H^1(K, \asd{}{c}{}{}{A})$. Note also that ${c}^*(\gamma)_v = 0$ for all $v\in S\smallsetminus\{v_0\}$ since $c_v = 0$ for these places, hence we can assume that $\alpha_v = 0$ for all $v \neq v_0$.

Consider now the image in $\prod_{S} \H^1(K_{v}, \asd{}{c}{}{}{E})$ of $(\alpha_v)_{v \in S}$. The twisting bijection $\H^1(\cdot, \asd{}{c}{}{}{E}) \to \H^1(\cdot, E)$ tells us that this element is the twist $(\asd{}{c}{}{}{\beta'_v})_{v \in S}$ of an element $(\beta'_v)_{v \in S} \in \prod_{S} \H^1(K_{v}, E)$. We claim that $(\beta'_v)_{v \in S}$ is not in the image of $\H^1(K, E)$.

\subsection*{Step 4: Compatibility between twists and obstructions.}
To prove the latter claim, we first show a certain compatibility relation between the maps $c^*$ and ${c'}^{*}$ for every two classes $c$ and $c'$  in $Z^1(K,G)=\H^1(K, G)$ such that $c_{v} = c'_{v}$ in $\H^1(K_{v}, G)$ for all $v\in S$.
Namely, we claim that there are canonical isomorphisms $i_{v} : \H^1(K_{v}, \asd{}{c}{}{}{A}) \xrightarrow{\sim} \H^1(K_{v},\asd{}{c'}{}{}{A})$ for $v\in S$ such that the following diagram
\begin{equation}\label{equ:compatible}
\xymatrix@R=1em{
& \Sha^1_S(K,\asd{}{c}{}{}{\hat A}) \times \prod_{S} \H^1(K_{v}, \asd{}{c}{}{}{A}) \ar[rr]^<<<<<<<<<<{\langle\, , \,\rangle_c} \ar@<8ex>[dd]^{\prod i_{v}}  & & \q / \z \ar@{=}[dd] \\
\Sha^1_\bic(G, \hat A) \ar[ru]^<<<<<<<<<<{c^*} \ar[rd]^<<<<<<<<<<{{c'}^*} & & & & \\
& \Sha^1_S(K, \asd{}{c'}{}{}{\hat A}) \times \prod_{S} \H^1(K_{v}, \asd{}{c'}{}{}{A}) \ar[rr]^<<<<<<<<<<{\langle \, , \,\rangle_{c'}} & & \q / \z ,
}
\end{equation}
is commutative in the following sense: for all $\gamma \in \Sha^1_\bic(G, \hat A)$ and all $(\alpha_{v})_{v \in S}$ in $\prod_{S} \H^1(K_v, \asd{}{c}{}{}{A})$, one has
\[\langle c^* \gamma, (\alpha_{v}) \rangle_c = \langle {c'}^* \gamma, (i_{v}(\alpha_{v})) \rangle_{c'} \in \q / \z .\]
The left-hand side arrows in the diagram were defined in Step 3, while the right-hand side arrows are given by Proposition \ref{Poitou-Tate lemma}. Now, for $v \in S$, since $c_v = c'_v$ as cocycles, there exists a natural isomorphism of $\gal(K_{v})$-modules $\asd{}{c}{}{}{A} \to \asd{}{c'}{}{}{A}$ (given by the identity map) that induces a group isomorphism
\[i_{v} : \H^1(K_{v}, \asd{}{c}{}{}{A}) \to \H^1(K_{v}, \asd{}{c'}{}{}{A}).\]
Similarly, one can consider the dual isomorphism $\asd{}{c'}{}{}{\hat A} \to \asd{}{c}{}{}{\hat A}$ which induces a corresponding isomorphism $\hat i_v:\H^1(K_v,\asd{}{c'}{}{}{\hat A})\ra \H^1(K_v,\asd{}{c}{}{}{\hat A})$ such that
\[\xymatrix@R=1em{
& \H^1(K_{v},  \asd{}{c}{}{}{\hat A}) \\
\Sha^1_\bic(G, \hat A)  \ar[ru]^{c_{v}^*} \ar[rd]^{{c'}_{v}^*} & \\
& \H^1(K_{v},  \asd{}{c'}{}{}{\hat A}) \ar[uu]^{\hat i_{v}} \, ,
}\]
is commutative. Hence the claim reduces to the commutativity of the following diagram
\[\xymatrix@R=1em{
\H^1(K_{v},\asd{}{c}{}{}{\hat A}) \times \H^1(K_{v}, \asd{}{c}{}{}{A}) \ar[rr]^<<<<<<<<<<{\langle \, ,\, \rangle_c} \ar@<8ex>[dd]^{i_{v}}  & & \q / \z \ar@{=}[dd] \\
& & & & \\
\H^1(K_{v}, \asd{}{c'}{}{}{\hat A}) \times \H^1(K_{v}, \asd{}{c'}{}{}{A}) \ar[rr]^<<<<<<<<<<{\langle\, , \,\rangle_{c'}} \ar@<7ex>[uu]^{\hat i_{v}} & & \q / \z \, ,
}\]
which is just the functoriality of cup-product (or of local duality), proving the claim.

\subsection*{Step 5: Proving that the local classes are globally unattainable.}
Assume on the contrary that there exists $\beta' = [b'] \in \H^1(K,E)$ lifting the $\beta'_{v}$'s for all $v \in S$. Define $c':=\pi(\beta')\in Z^1(K,G)=\H^1(K,G)$ and consider the class $\beta'':=s(\pi(\beta'))=[c']\in \H^1(K,E)$ (we are once again doing an abuse of notation with $c'$).

Let us now twist the exact sequence \eqref{ses} by the cocycle $c'$, so that we get the following twisted exact sequence:
\begin{equation} \label{ses3}
\xymatrix{
1 \ar[r] & \asd{}{c'}{}{}{A} \ar[r] & \asd{}{c'}{}{}{E} \ar[r] & \asd{}{c'}{}{}{G} \ar[r] \ar@/_1pc/[l] & 1.
}
\end{equation}
where we remark once again that $\asd{}{c'}{}{}{G}=G$ and that the section $s$ is still well defined on these twisted forms.

Since $\beta'$ and $\beta''=[c']$ have the same image in $\H^1(K,G)$, we know that $\asd{}{c'}{}{}{\beta'}$ must come from an element $\alpha'\in \H^1(K,\asd{}{c'}{}{}{A})$ and that its image $(\alpha'_v)_{v\in S}\in \prod_S \H^1(K,\asd{}{c'}{}{}{A})$ maps to $(\asd{}{c'}{}{v}{\beta'})_{v\in S}\in \prod_S \H^1(K,\asd{}{c'}{}{}{E})$. In particular, since $(\alpha'_v)_{v\in S}$ comes from a global element, we know by Proposition \ref{Poitou-Tate lemma} that
\begin{equation}\label{equ:ort}
(\alpha'_v)_{v\in S} \textbf{ is orthogonal to } \Sha^1_S(K,\asd{}{c'}{}{}{\hat A}) \, .
\end{equation}

Consider now Diagram \eqref{equ:compatible} and the element $(\alpha_v)_{v\in S}\in \prod_S \H^1(K_v,\asd{}{c}{}{}{A})$ from Step~3. By construction, this element is not orthogonal to ${c}^*(\gamma)$ and hence its image $ (i_v(\alpha_v))_{v\in S}\in \prod_S \H^1(K,\asd{}{c'}{}{}{A})$ is not orthogonal to ${c'}^*(\gamma)$.

The two elements $(i_v(\alpha_v))_{v\in S}$ and $(\alpha'_v)_{v\in S}$ have the same image in $\prod_{v \in S} \H^1(K_{v},\asd{}{c'}{}{}{E})$ by construction (recall that $c_v=c'_v$ for $v\in S$ and hence $\asd{}{c}{}{v}{\beta'}=\asd{}{c'}{}{v}{\beta'}$).
Thus, $(i_v(\alpha_v))_{v\in S}$ and $(\alpha'_v)_{v\in S}$ differ by an action (as described in \eqref{equ:H0-act}, \S\ref{section twisting}) of an element of the group
\[\coker \left(\prod_{v \in S} \H^0(K_{v}, \asd{}{b'}{}{}{E}) \to \prod_{v \in S} \H^0(K_{v},\asd{}{c'}{}{}{G}) \right) .\]
We denote by $({g}_{v})_{v \in S}\in \prod_{S} \H^0(K_{v},\asd{}{c'}{}{}{G})$ a lift of this element.

Now, for all $v \in S\smallsetminus\{v_0\}$, we have by construction $\beta'_v = 1 \in \H^1(K_v, E)$ and hence $b'$ (and a fortiori $c'$) is trivial over $\gal(K_v)$ since $E$ is constant. The map $\H^0(K_{v}, \asd{}{b'}{}{}{E}) \to \H^0(K_{v},\asd{}{c'}{}{}{G})$ corresponds then to $E\to G$, which is surjective. Therefore, $i_v(\alpha_v) = \alpha'_v$ and $\H^0(K_v,\asd{}{c'}{}{}{G})$ acts trivially on $i_v(\alpha_v)$ for these $v$. Finally, the map $\H^0(K,\asd{}{c'}{}{}{G}) \to \H^0(K_{v_0},\asd{}{c'}{}{}{G})$ is clearly surjective (recall that $\asd{}{c'}{}{}{G}=G$ and hence both groups are equal to $G$ and the map is the identity), hence ${g}_{v_0}$ lifts to an element ${g} \in \H^0(K, G)$. We conclude then that $g\cdot (i_v(\alpha_v))_{v\in S}=(\alpha'_{v})_{v \in S}$.

Note that the natural paring $\asd{}{c'}{}{}{A}\times \asd{}{c'}{}{}{\hat A} \ra \mathbb G_m$ is invariant under the morphism $(g,\hat g^{-1}):\asd{}{c'}{}{}{A}\times \asd{}{c'}{}{}{\hat A}\ra \asd{}{c'}{}{}{A}\times \asd{}{c'}{}{}{\hat A}$ given by $(a,f)\ra (g\cdot a,\hat g^{-1}\cdot f)$,  where
$\hat{g} : \asd{}{c'}{}{}{\hat A} \to \asd{}{c'}{}{}{\hat A}$ 
is the dual map to 
${g} : \asd{}{c'}{}{}{A} \to \asd{}{c'}{}{}{A}$ 
defined by the conjugation action of $G$ on $\asd{}{c'}{}{}{A}$. 
By functoriality of cup products \cite[Proposition 1.4.2]{NSW}, the cup product 
$\H^1(K_v, \asd{}{c'}{}{}{A})\times \H^1(K_v,\asd{}{c'}{}{}{\hat A}) \to \mathbb Q/\mathbb Z$ is also invariant under the morphism induced by $(g,\hat g^{-1})$.
Since $(i_v(\alpha_v))_{v\in S}$ is not orthogonal to ${c'}^*(\gamma)$, it follows that 
${g} \cdot (i_v(\alpha_v))_{v\in S}$ is not orthogonal to $\hat{g}^{-1}({c'}^*(\gamma))$. But as ${c'}^*(\gamma)$ is in $\Sha^1_S(K,\asd{}{c'}{}{}{\hat A})$ so is $\hat{g}^{-1}({c'}^*(\gamma))$, hence we finally get that \[g\cdot (i_v(\alpha_v))_{v\in S}=(\alpha'_{v})_{v \in S} \textbf{ is not orthogonal to } \Sha^1_S(K,\asd{}{c'}{}{}{\hat A}),\]
contradicting \eqref{equ:ort}.
\end{proof}


\begin{thebibliography}{ABC00a}

\bibitem[Col03]{CT}
{\bf J.-L. Colliot-Th\'el\`ene.} Points rationnels sur les fibrations. {\it Higher Dimensional Varieties and Rational Points}, Bolyai Society Mathematical Series, vol 12, Springer-Verlag, edited by K. J. B\"or\"oczky,  J. Koll\`ar and  T. Szamuely, 171--221, 2003.

\bibitem[Col14]{ColliotBrnr}
{\bf Jean-Louis Colliot-Th{\'e}l{\`e}ne.} Groupe de {B}rauer non ramifi\'e de quotients par un groupe fini. {\it {P}roc. {A}mer. {M}ath. {S}oc.} 142(5), 1457-1469, 2014.

\bibitem[DG12]{DG}
{\bf P. D\`ebes, N. Ghazi.} Galois covers and the Hilbert-Grunwald property. {\it Ann. Inst. Fourier} 62, 989--1013, 2012.

\bibitem[Dem10]{Demarche}
{\bf C.~Demarche.} Groupe de {B}rauer non ramifi\'e d'espaces homog\`enes \`a stabilisateurs finis. {\it Math. Ann.} 346(4), 949--968, 2010.

\bibitem[Gru33]{Gru}
{\bf W.~Grunwald.} Ein allgemeines Existenztheorem f\"ur algebraische Zahlk\"orper. {\it J. Reine Angew. Math.} 169, 103--107, 1933.

\bibitem[Har07]{HarariBulletinSMF}
{\bf D.~Harari.} Quelques propri\'et\'es d'approximation reli\'ees \`a la cohomologie galoisienne d'un groupe alg\'ebrique fini. {\it Bull. Soc. Math. France} 135(4), 549--564, 2007.

\bibitem[Kal48]{Kal}
{\bf L.~Kaloujnine.} La structure des $p$-groupes de Sylow des groupes sym\'etriques finis. {\it Ann. Sci. \'Ecole Norm. Sup. (3)} 65, 239--276, 1948.

\bibitem[Luc14a]{GLA-AFPHEH}
{\bf G.~{L}ucchini {A}rteche.} Approximation faible et principe de Hasse pour des espaces homog\`enes \`a stabilisateur fini r\'esoluble. {\it Math. Ann.} 360, 1021--1039, 2014.

\bibitem[Luc14b]{GLABrnral1}
{\bf G.~{L}ucchini {A}rteche.} Groupe de Brauer non ramifi\'e des espaces homog\`enes \`a stabilisateur fini. {\it J. Algebra} 411, 129-181, 2014.

\bibitem[Neu79]{NeukirchSolvable}
{\bf J.~Neukirch.} On solvable number fields. {\it Invent. Math.} 53(2), 135--164, 1979.

\bibitem[NSW08]{NSW}
{\bf J.~Neukirch, A. Schmidt, K. Wingberg.} {\it Cohomology of number fields}. Grundlehren der Mathematischen Wissenschaften, No.~323.
Springer-Verlag, Berlin, second edition, 2008.

\bibitem[Pie82]{Pie}
{\bf R.~Pierce.} {\it Associative algebras}. Graduate Texts in Mathematics, No.~88. Springer-Verlag, New York-Berlin, 1982.

\bibitem[Sal82]{Sal}
{\bf D.~Saltman.} Generic Galois extensions and problems in field theory. {\it Adv. in Math.} 43, 250--283, 1982.

\bibitem[Ser02]{SerreCohGal}
{\bf J.-P.~Serre.} {\it Galois cohomology}. Springer Monographs in Mathematics. Springer-Verlag, Berlin, {E}nglish edition, 2002.

\bibitem[Sko01]{Skor}
{\bf A.~Skorobogatov.} {\it Torsors and rational points}, Cambridge Tracts in Mathematics, No.~144. Cambridge University Press, Cambridge, 2001.

\bibitem[Wan50]{Wan}
{\bf S.~Wang.} On Grunwald's theorem. {\it Ann. of Math. (2)} 51, 471--484, 1950.

\bibitem[Wei55]{Weir}
{\bf A.~J.~Weir.} Sylow $p$-subgroups of the classical groups over finite fields with characteristic prime to $p$. {\it Proc. Amer. Math. Soc.} 6, 529--533, 1955.


\end{thebibliography}
\end{document}